\documentclass{amsart}
\usepackage{amssymb}
\usepackage{amsbsy}
\usepackage{amscd}
\usepackage{amsmath}
\usepackage{amsthm}
\usepackage{amsxtra}
\usepackage{latexsym}
\usepackage[mathscr]{eucal}
\usepackage{upgreek}
\usepackage{wasysym}
\usepackage[all,cmtip]{xy}

\usepackage{xcolor}
\definecolor{rouge}{rgb}{0.85,0.1,.4}
\definecolor{bleu}{rgb}{0.1,0.2,0.9}
\definecolor{violet}{rgb}{0.7,0,0.8}
\usepackage[dvipdfmx,colorlinks=true,linkcolor=bleu,urlcolor=violet,citecolor=rouge]{hyperref}

\newcommand{\bra}{{\langle}}
\newcommand{\ket}{{\rangle}}
\newcommand{\BRS}[1]{H^{\frac{\infty}{2}+0}_{f_{\theta}}(#1)}
\newcommand{\Lam}{\Lambda}

\newcommand{\on}{\operatorname}
\newcommand{\+}{\mathop{\oplus}}
\renewcommand{\*}{{\otimes}}
\newcommand{\mc}{\mathcal}
\newcommand{\mf}{\mathfrak}

\newcommand{\fing}{\mf{g}}

\newcommand{\affg}{\widehat{\mf{g}}}

\newcommand{\isomap}{{\;\stackrel{_\sim}{\to}\;}}
\newcommand{\Z}{\mathbb{Z}}
\newcommand{\C}{\mathbb{C}}

\newcommand{\W}{\mathscr{W}}
\newcommand{\ra}{\rightarrow}
\newcommand{\lam}{\lambda}

\def\leq{\leqslant}
\def\geq{\geqslant}

\DeclareMathOperator{\End}{End}

\DeclareMathOperator{\gr}{gr}

\DeclareMathOperator{\ad}{ad}

\theoremstyle{theorem}
\newtheorem{Th}{Theorem}[section]

\newtheorem{Pro}[Th]{Proposition}

\newtheorem{Lem}[Th]{Lemma}

\newtheorem{Co}[Th]{Corollary}

\theoremstyle{remark}

\newtheorem{Rem}[Th]{Remark}
\newtheorem{Conj}{Conjecture}

\title{Joseph ideals and lisse minimal $W$-algebras}
\author{Tomoyuki Arakawa}
\address{Research Institute for Mathematical Sciences, Kyoto University,
 Kyoto 606-8502 JAPAN}
\email{arakawa@kurims.kyoto-u.ac.jp}
\author{Anne Moreau}
\address{Laboratoire de Math\'{e}matiques et Applications, T\'{e}l\'{e}port 2 - BP 30179, Boulevard Marie et Pierre Curie, 86962 Futuroscope Chasseneuil Cedex, France}
\email{anne.moreau@math.univ-poitiers.fr}

\begin{document}
\maketitle

\begin{abstract}
 We consider a lifting of   Joseph ideals for the minimal nilpotent orbit
 closure to the setting of 
 affine Kac-Moody algebras and  find  new examples of affine vertex
 algebras
 whose associated varieties are minimal nilpotent orbit closures.
 As an application we obtain a new family  of  lisse ($C_2$-cofinite) $W$-algebras
 that are not coming from admissible representations of affine Kac-Moody algebras.
\end{abstract}

\section{Introduction}
Let  $\fing$ be a finite-dimensional simple Lie algebra over $\C$,
$(~|~)$ be a normalized invariant inner product,
i.e., $\frac{1}{2h^{\vee}}\times$Killing form.

Let $\affg=\fing[t,t^{-1}]\+ \C K \+ \C D$ be the Kac-Moody Lie algebra
associated with $\fing$ and $(~|~)$, with the commutation relations
\begin{align*}
 [x(m),y(n)]=[x,y](m+n)+ m(x|y)\delta_{m+n,0}K,\\
[D,x(m)]=mx(m),\quad [K,\affg]=0,
\end{align*}
where $x(m)=x\otimes t^m$.
For $k\in \C$, set
\begin{align*}
 V^k(\fing)=U(\affg)\*_{U(\fing[t]\+ \C K \+ \C D)}\C_k,
\end{align*}
where 
$\C_k$ is the one-dimensional representation of
$\fing[t]\+ \C K \+ \C D$ on which 
$\fing[t]\+ D$ acts trivially and $K$ acts as multiplication by $k$.
The space $V^k(\fing)$ is naturally a vertex algebra,
and it is called the {\em universal affine vertex algebra associated with
$\fing$
at level $k$}. 
By the PBW theorem, $V^k(\fing) \cong U(\fing[t^{-1}]t^{-1})$ as $\C$-vector spaces. 

Let $V_k(\fing)$ be the unique simple graded quotient of $V^k(\fing)$.
As a $\affg$-module,
$V_k(\fing)$ is isomorphic to the irreducible highest weight
representation
of $\affg$ with highest weight $k\Lam_0$, 
where $\Lam_0$ is the dual element of $K$. %highest weight of the vacuum representation. 

Let $X_V$ be the {\em associated variety} \cite{Ara12}
of a vertex algebra $V$,
which is the maximum spectrum of {\em Zhu's $C_2$-algebra}, 
$$ R_V := V/C_2(V).$$
In the case $V$ is a quotient of $V^k(\fing)$,
$V/C_2(V)=V/\fing[t^{-1}]t^{-2} V$ 
and we have a surjective Poisson algebra homomorphism
\begin{align*}
 \C[\fing^*]=S(\fing)\twoheadrightarrow V/\fing[t^{-1}]t^{-2}V,\quad
x\mapsto \overline{x(-1)}+\fing[t^{-1}]t^{-2}V,
\end{align*}
where $\overline{x(-1)}$ denotes the image of $x(-1)$ in the quotient 
$V$. 
Then $X_V$ is just the zero locus of the kernel of the above map in
$\fing^*$. It is
$G$-invariant and conic, where 
$G$ is the adjoint group of $\fing$.
Note that on the contrary to
the associated variety of a primitive ideal of $U(\fing)$, 
the variety $X_{V_k(\fing)}$ is not necessarily
contained in the nilpotent cone $\mc{N}$ of $\fing$.
In fact,  
$X_{V_k(\fing)}=\fing^*$ for a generic $k$ since 
$V_k(\fing)=V^k(\fing)$ in this case.

A conjecture of Feigin and Frenkel, proved in \cite{Ara09b},
states  that $X_{V_k(\fing)}\subset \mc{N}$ if
$V_k(\fing)$ is admissible \cite{KacWak89}.
In fact it is also believed that  the converse is true, that is,
$X_{V_k(\fing)}\subset \mc{N}$
only if
$V_k(\fing)$ is admissible, so that the condition 
$X_{V_k(\fing)}\subset \mc{N}$
gives a geometric characterization of admissible affine vertex algebras. 
%representations of $\affg$.
One of the aims of this paper is to provide a 
counterexample of this fact, that is, there exist non-admissible
affine vertex algebras $V_k(\fing)$ such that
$X_{V_k(\fing)}\subset \mc{N}$.

Let $(e_\theta,h_\theta,f_\theta)$ be the $\mf{sl}_2$-triple associated with 
the highest positive root $\theta$ of $\fing$. 
Let $\mathbb{O}_{min}=G.f_\theta$ be the unique minimal non-trivial nilpotent
  orbit of $\fing$ which is of dimension $2h^\vee-2$, \cite{Wa99},  
where $h^{\vee}$ is the dual Coxeter number of $\fing$. 

%Let $\mathbb{O}_{min}$ be the unique minimal non-trivial nilpotent
%  orbit of $\fing$, and let $h^{\vee}$ be the dual Coxeter 
%number of $\fing$. 
  
Consider the {\em Deligne exceptional series} 
\begin{align*}
 A_1\subset A_2\subset G_2\subset D_4\subset F_4\subset E_6\subset
 E_7\subset E_8
\end{align*}
discussed in  \cite{De96,DeGr02}.
   
\begin{Th} \label{Th1} 
\begin{enumerate}
\item Assume that $\fing$ belongs to the Deligne exceptional series and that 
	  \begin{align*}
	   k= -\frac{h^{\vee}}{6}-1. 
	  \end{align*}
Then $X_{V_{k}(\fing)}=\overline{\mathbb{O}_{min}}.$

\item Assume that $\fing$ is of type $D_4$, $E_6$, $E_7$, $E_8$ and that 
       %\marginpar{\tiny{In type $A_1$, $A_2$, $G_2$, $F_4$ $k$ is not an
       %integer, so how to describe..Do you have an idea?\color{red}is it
      %better? YES!}}
	$k$ is an integer such that
	  \begin{align*}
	   -\frac{h^{\vee}}{6}-1
	   \leq k\leq -1. 
	  \end{align*}
Then $ X_{V_{k}(\fing)}=\overline{\mathbb{O}_{min}}$.
\item Assume that $\fing$ is of type $D_l$, $l\geq 5$.
Then
$ X_{V_{k}(\fing)}=\overline{\mathbb{O}_{min}}$ for $k=-2,-1$.
\end{enumerate}  
 \end{Th} 

%We conjecture that 
%if 
%$ X_{V_{k}(\fing)}=\overline{\mathbb{O}_{min}}$
% then 
%$\fing$ and $k$ satisfy the above conditions (2) or (3) 
%unless
% $k$ is admissible, or  $\fing=\mf{sl}_2$ and $k=-2$, see Conjecture \ref{Conj:when-variety-is-the-minimal-nilpotent-orbit-closure}.
Note that for $\fing$ of type $A_1$, $A_2$, $G_2$, $F_4$, the rational number 
 $-{h^{\vee}}/{6}-1$ is admissible. 
However for types $D_4$, $E_6$, $E_7$, $E_8$, 
the number  $-{h^{\vee}}/{6}-1$  is a negative integer which is
certainly non-admissible (\cite[Proposition 1.2]{KacWak08}).

%In the case that $k$ is admissible 
%the vertex algebra 
%$V_k(\fing)$  has
%interesting representation theory see e.g., 

%The study of the representation $V_k(\fing)$ intensively started with works of 
%Frenkel-Zhu and Adamovi{\'c} among others; 

A consequence of the fact $X_{V_k(\fing)}\subset \mc{N}$ is that
$V_k(\fing)$ has only finitely many simple modules in the category
$\mc{O}$ (cf.~Corollary \ref{cor:finitely-modules}), as in case
$V_k(\fing)$ is admissible \cite{Ada94,AdaMil95,Ada97,Per07,Per07typeF,AxtLee11,A12-2}.
If $\fing$ belongs to the Deligne exceptional series
outside the type $A$ 
and $k= -h^{\vee}/6-1$,
 it is possible to derive
 the classification of simple $V_k(\fing)$-modules that belong to
 $\mc{O}$
 from Joseph's result \cite{Jos98} in the following manner.

If $\fing$ is not of type $A$, 
it is known \cite{Joseph:1976kq,GanSav04} that there exists a unique completely prime ideal
$\mc{J}_0$ in $U(\fing)$, 
called the {\em Joseph ideal}, whose associated variety is
$\overline{\mathbb{O}_{min}}$, that is, $\overline{\mathbb{O}_{min}}$ 
is the zero locus in $\fing^*$ of ${\rm gr}\,\mc{J}_0$. 
As a by-product, we obtain a lifting to the Joseph ideal in the following sense. 
For a $\Z_{\geq 0}$-graded vertex algebra $V$, let $A(V)$ be its Zhu's
algebra \cite{Zhu96}. 
Such a vertex algebra $V$ is called a {\em chiralization} of an algebra $A$ if
$A(V)\cong A$. 
We claim that if $\fing$
belongs to the {Deligne exceptional series} outside the type $A$ 
and if $k=- h^\vee/6-1$, then $V_k(\fing)$ is a chiralization of $\C\oplus U(\fing)/\mc{J}_0$. 
Namely, 
$$A(V_k(\fing)) \cong U(\fing)/\mc{J}_{\mc{W}} \cong \C\times U(\fing)/\mc{J}_0,
%\color{black}
$$
%\marginpar{\tiny{\color{red}I replaced everywhere $\oplus$ with $\times$ for algebras}}
for some ideal $\mc{J}_{\mc{W}}$ (cf.~Proposition \ref{Pro:enlarged-Joseph} 
and Theorem \ref{Th:lifting-of-Joseph-ideal}). 
%We cannot expect that $V_k(\fing)$ is a chiralization of $U(\fing)/\mc{J}_0$ 
%since $L_{\fing}(0) \simeq \C$ is always a module of $A(V_k(\fing)) \cong A (L(k\Lambda_%0))$, 
%but not of $U(\fing)/\mc{J}_0$ since its associated variety is nonzero. 
%\marginpar{\tiny Not sure this sentence is understandable without what you explain about% the top 
%degree: remove it?}
%Here, $L(\lambda)$ (resp.~$L_{\fing}(\lambda)$) denotes the irreducible representation of 
%$\affg$ (resp.~$\fing$) with highest weight $\lambda$. 
Hence
the classification of simple highest weight $U(\fing)/\mc{J}_0$-modules 
obtained
in \cite{Jos98} gives 
the classification of simple highest weight $V_k(\fing)$-modules thanks
to  Zhu's
theorem \cite{Zhu96},
which for types $G_2$, $D_4$,  $F_4$
reproves the earlier results
obtained in 
 \cite{AxtLee11} and
\cite{Per07typeF,Per13}.

Another consequence
  of the fact $X_{V_k(\fing)}\subset \mc{N}$ is  that the $D$-module on the moduli stack of
$G$-bundles on a curve obtained from $V_k(\fing)$ by the Harish-Chandra
localization \cite{BeiDri96,FreBen04} has its micro-local support inside 
the global nilpotent cone.
It would be very interesting to  consider the associated modular
functor (cf.\ \cite{FeiMal97}),
or
the corresponding conformal field
theory (cf.\ \cite{CreRid12,CreRid13}).
We hope to come back to this point in our future work.

In physics  literature 
the affine vertex algebras
in Theorem \ref{Th1} (1)
have been studied  in the work \cite{BeeLemLie15}
of Beem, Lemos, Liendo, Peelaers, Rastelli and van Rees
in connection with {\em four} dimensional superconformal field theory.
The associated varieties of these vertex algebras  seem to describe
 the Higgs  branch of the corresponding four dimensional theory.
We also hope to come back to this point in our future work.

\smallskip

Theorem \ref{Th1}, or its proof, has the following important application:

Let $\W^k(\fing,f_{\theta})$ be the {\em $W$-algebra associated with}
$(\fing,f_{\theta})$ at level $k$ \cite{KacRoaWak03},
which is a conformal vertex algebra with central charge
\begin{align*}
c(k)= \frac{k\dim \fing}{k+h^{\vee}}-6k+h^{\vee}-4
\end{align*}
provided that
$k\ne -h^{\vee}$.
Note that
if $\fing$ belongs to the Deligne exceptional series,
\begin{align*}
c(k)= -\frac{6(k+h^{\vee}/6+1)((h^{\vee}/6+1)k-
 (h^{\vee}-4)h^{\vee}/6)}{(k+h^{\vee})(h^{\vee}/6+1)
 },
\end{align*}
 so that $c(k)=0$  for $k=-h^{\vee}/6-1$.

Denote by $\W_k(\fing,f_{\theta})$ the unique simple  quotient
of $\W^k(\fing,f_\theta)$.
Since
$X_{\W^k(\fing,f_{\theta})}$ is naturally isomorphic to the Slodowy slice $\mc{S}_{min}$ 
at $f_{\theta}$ (\cite{De-Kac06,Ara09b}), with 
\begin{align*} 
\mc{S}_{min} := f_\theta + \fing^{e_\theta}, 
\qquad  \fing^{e_\theta} = \{x\in\fing \; |\; 
[x,e_\theta]=0\},
\end{align*} 
%\marginpar{\tiny{Where it is stated in \cite{Ara09b}? 
%\color{red}
%In Corollary 4.18.
%\color{black}We don't need to take 
%the intersection with $\mathcal{N}$ or $\overline{G.f_{\theta}}$?
%\color{red}No.}}
%\marginpar{\tiny{Also, I don't see anything about this in
%\cite{De-Kac06}.
%\color{red}
%I don't remember where, but there is an equivalent statement there.}}
the variety $X_{\W_k(\fing,f_{\theta})}$ is a $\C^*$-invariant, Poisson
subvariety of $\mc{S}_{min}$.

It is known \cite{De-Kac06}
that the 
 (Ramond twisted) Zhu's algebra of $\W^k(\fing,f_{\theta})$ is naturally
isomorphic to the finite $W$-algebra
$U(\fing,f_{\theta})$ associated with $(\fing,f_{\theta})$ introduced by
 Premet \cite{Pre02}.

Premet \cite{Pre07}  has shown that 
the Joseph ideal is closely connected with 
one-dimensional representations of $U(\fing,f_{\theta})$.
The chiralization of $U(\fing)/\mc{J}_W$ explained above is 
closely related with
one-dimensional
representations of $\W^k(\fing,f_{\theta})$ as well.
The significant difference in the affine setting is that
$\W^k(\fing,f_{\theta})$ does not necessarily admit one-dimensional
representations.
In fact $\W^k(\fing,f_{\theta})$, $\fing\ne \mf{sl}_2$, admits one-dimensional representations
if and only if $\W_k(\fing,f_{\theta})=\C$, and this happens if and only if
$\fing$ belongs to the Deligne exceptional series
and $k=-h^{\vee}/6-1$, or 
$\fing$ is of type $C_l$ and $k=-1/2$ (cf.~Theorem \ref{TH:trivial-rep}). 

Note that
the trivial vertex algebra $\C$ is certainly a {\em lisse} vertex algebra.
Here,
recall that a vertex algebra $V$ is called lisse,
or {\em $C_2$-cofinite}, if 
$\dim X_V=0$.
Lisse vertex algebras may be regarded as an analogue of finite-dimensional algebras.
One of remarkable properties of a lisse vertex algebra $V$ is the modular
invariance of characters of modules \cite{Zhu96,Miy04}.
Further, if it is non-trivial and also {\em rational}, it is known \cite{Hua08rigidity}
that under some mild assumptions the category of $V$-modules forms a modular tensor category,
which for instance yields an invariant of $3$-manifolds, see \cite{Bakalov:2001kq}.

In  \cite{Ara09b},  
in order to approach
the Kac-Wakimoto conjecture \cite{KacWak08}
on the rationality of {\em exceptional $W$-algebras},
the first named author showed that each admissible affine vertex algebra
produces exactly one lisse simple $W$-algebra.
 More precisely, 
the associated variety of an admissible affine vertex algebra
$V_k(\fing)$
is 
isomorphic to $\overline{\mathbb{O}}$ for some nilpotent orbit
$\mathbb{O}$ of $\fing$,
and if we take the nilpotent element $f$ from this orbit $\mathbb{O}$, 
then $\W_k(\fing,f)$ is lisse.
Until very recently it has been  widely believed that these
$W$-algebras are all the lisse $W$-algebras, cf. \cite{KacWak08}.
However, it turned out that there are a lot more.

\begin{Th}\label{Th:new-lisse}
 \begin{enumerate}
  \item Let $\fing$ be of type $D_4$, $E_6$, $E_7$, $E_8$.
	 For any integer $k$ 
	that is equal to or greater than $-h^{\vee}/6-1$,
	the simple $W$-algebra
 $\W_k(\fing,f_{\theta})$ is lisse.
  \item Let $\fing$ be of type $D_l$ with $l\geq 5$.
	For any integer $k$ 
	that is equal to or greater than $-2$,
the simple $W$-algebra  $\W_k(\fing,f_{\theta})$ is lisse.

 \end{enumerate}
\end{Th}

In the case that $k=-h^{\vee}/6$,
the first statement of Theorem \ref{Th:new-lisse} is a recent result of
Kawasetsu \cite{Kawa15}.
Kawasetsu  actually proved that  $\W_{-h^{\vee}/6}(\fing,f_{\theta})$ is
rational and
$C_2$-cofinite if $\fing$ belongs to the Deligne exceptional series,
providing a first (surprising) example of rational and $C_2$-cofinite $W$-algebras
that are not coming from admissible representations of $\affg$.
Our present work is motivated by his result.
It would be very
interesting to know whether the lisse $W$-algebras appearing in Theorem
\ref{Th:new-lisse}
are rational or not.
We hope to come back to this point in future work.

\subsection*{Acknowledgments.}
A part of this work was done while the first named author was staying at
the Universit\'{e} de Poitiers in October, 2014
and at the Centro di Ricerca Matematica Ennio De Giorg in
Pisa in December, 2014 and in January, 2015.
He would like to thank both institutes.
He would also like to thank Kazuya Kawasetsu for useful
discussions, and  Dra{\v{z}}en Adamovi{\'c}
for bringing the article \cite{Per13} to his attention.
After submitting the first version of the present paper he had stimulating discussions with
Leonardo Rastelli, Hiraku Nakajima, Takahiro Nishinaka and Yuji Tachikawa.
He would like to thank all of them.
His research is supported by JSPS KAKENHI Grant Numbers 25287004 and 26610006.

The second named author would like to thank Rupert Wei Tze Yu for bringing 
the article \cite{GanSav04} to her attention, and Pierre Torasso for useful discussions 
about central characters. 
Her research is supported by the ANR Project GERCHER Grant number ANR-2010-BLAN-110-02. 

Both authors thank Dra{\v{z}}en Adamovi{\'c}, Victor Kac, Ozren Per{\v{s}}e and Weiqiang Wang 
for their comments on the first version of this article. 

\section{Minimal nilpotent orbit closures and Joseph ideals}
Let $J_0$ be the prime ideal of $S(\fing)$ 
corresponding to the minimal nilpotent orbit closure $\overline{\mathbb{O}_{min}}$ 
in $\fing^*$. 

Suppose that $\fing$ is not of type $A$. 
According to Kostant,
$J_0$ is generated %\marginpar{\tiny{Is this true for type  $A$? \color{red}I don't think so: I have changed}}
by a $\fing$-submodule $L_{\fing}(0)\+ W$ in $S^2(\fing)$,
such that
\begin{align*}
 S^2(\fing)=L_{\fing}(2\theta)\+ L_{\fing}(0)\+ W,
\end{align*}
where $L_{\fing}(\lam)$ is the irreducible   representation of $\fing$ with highest
weight $\lam$
and $\theta$ is the highest root of $\fing$. 

%\marginpar{\tiny{Do you like the notation $E(\lambda)$? 
%Better to use $L_{\fing}(\lambda)$ as later?\color{red} changed!}}
Note that the above decomposition of $S^2(\fing)$ 
still holds in type $A$, \cite[Chapter IV, Proposition 2]{Gar}, 

%\begin{Rem} 
%By \cite[Definition 1 page 29]{Gar}, 
%$J_0 = \bigoplus_{k \geq 2} W_{k}$ if $S^k(\fing) = L_{\fing}(k \theta) \oplus W_{k}$ 
%where $W_k$ is the $\fing$-invariant complement to $L_{\fing}(k \theta)$. 
%But I am not sure that $J_0$ is generated by elements of degree 2 for the type $A$... 
%The proof of \cite{Gar} to claim this in the other types does not holds in type $A$. 
%
%Also, \cite[Theorem 8.2.5]{Wey} suggests that $J_0$ is probably not generated 
%by elements of degree 2. \color{red}I see! \color{black}
%\marginpar{\tiny Then remove the remark?}
%\end{Rem}

Also, note that
 $L_{\fing}(0)=\C\Omega$ 
where $\Omega$ is the Casimir element in $S(\fing)$.
 \begin{Lem}\label{Lem:Gan-Savin}
Suppose that $\fing$ is not of type $A$.
The ideal $J_W$ in $S(\fing)$ generated by $W$ contains 
  $\Omega^2$, and hence, $\sqrt{J_W}=J_0$.
 \end{Lem}
 \begin{proof}By 
the proof of \cite[Theorem 3.1]{GanSav04}
$J_W$ contains $\fing \cdot \Omega$,  and the assertion follows.
 \end{proof}

The structure of $W$ was determined by Garfinkle \cite{Gar}.
%Let $\{e_{\theta},f_{\theta},h_{\theta}\}$ be the standard
%$\mf{sl}_2$-triple
%corresponding to $\theta$.
%Then %the orbit 
%$\mathbb{O}_{min}$ is the $G$-orbit of $f_{\theta}$.
Set
\begin{align*}
\fing(j)=\{x\in \fing\mid [h_{\theta},x]=2 jx\}.
\end{align*}
Then 
\begin{align*}
 \fing&=\fing({-1})\+\fing(-1/2)\+\fing({0})\+\fing(1/2) 
 \+ \fing(1),\\& \fing({-1})=\C f_{\theta},
 \ \fing(1)=\C e_{\theta},\
 \fing(0)=\C h_{\theta}\+ \fing^{\natural},
 \ \fing^{\natural}=\{x\in \fing(0)\mid (h_{\theta}|x)=0\}.
\end{align*}
The subalgebra
$\fing^{\natural}$ is a reductive subalgebra of $\fing$
whose simple roots are the simple roots of $\fing$
perpendicular to $\theta$.
Write 
\begin{align*}
[\fing^{\natural},\fing^{\natural}]=\bigoplus_{i\geq 1} \fing_i
\end{align*}
as a
direct sum of simple summands, and let $\theta_i$ be the highest root of
$\fing_i$.

If $\fing$ is neither of type $A_l$ nor $C_l$,
\begin{align*}
 W=\bigoplus_{i\geq 1}L_{\fing}(\theta+\theta_i).
\end{align*}

If $\fing$ is of type $C_l$,
then $\fing^{\natural}$
 is simple of type $C_{l-1}$,
so that there is a unique $\theta_1$,
and we have
\begin{align*}
 W=L_{\fing}(\theta+\theta_1)\+ L_{\fing}(\frac{1}{2}(\theta+\theta_1)).
\end{align*}

%If $\fing$ is of type $A_l$, $l\geq 3$,
%then  $[\fing^{\natural},\fing^{\natural}]$ is simple of type $A_{l-2}$,
%so that there is a unique $\theta_1$, and we have 
%%\marginpar{\tiny{Shall
%%we remove the  type $A$ case?}}
%\begin{align*}
% W=L_{\fing}(\theta)\+ L_{\fing}(\theta+\theta_1)
%\end{align*}
%
%If $\fing$ is of type $A_2$, 
%\begin{align*}
% W=L_{\fing}(\theta).
%\end{align*} 
%
%If $\fing$ is of type $A_1$, 
%\begin{align*}
% W= \{0\}.
%\end{align*} 

If $\fing$ is not of type $A$, 
it is known \cite{Joseph:1976kq,GanSav04} that there exists a unique completely prime ideal
$\mc{J}_0$ in $U(\fing)$,
called the {\em Joseph ideal}, whose associated variety is
$\overline{\mathbb{O}_{min}}$. 
It is known that $\mc{J}_0$ is 
maximal and  {\em primitive}.
By \cite{Gar,GanSav04} 
$\mc{J}_0$ is generated by $W$ and $\Omega-c_0$,
where $W$ is identified with a $\fing$-submodule of $U(\fing)$ by the
$\fing$-module isomorphism
$S(\fing)\cong U(\fing)$ and 
$c_0$ is the eigenvalue of $\Omega$ for the infinitesimal character that
Joseph obtained in \cite[Table p.15]{Joseph:1976kq}.
We have \begin{align*}
      \gr{\mc{J}_0}=J_0=\sqrt{J_W}
     \end{align*}
and this shows that $\mc{J}_0$ is indeed completely prime.

Let $\mc{J}_W$ be the two-sided ideal of $U(\fing)$ generated by $W$.
 \begin{Pro}\label{Pro:enlarged-Joseph}
We have an  algebra isomorphism
 \begin{align*}
U(\fing)/\mc{J}_W\cong \C \times U(\fing)/\mc{J}_0. 
 \end{align*}
 \end{Pro}
  \begin{proof}
By the proof of \cite[Theorem 3.1]{GanSav04},
$\mc{J}_W$ contains $(\Omega-c_0)\fing$.
Hence it contains $(\Omega-c_0)\Omega$.
Since $c_0\ne 0$,
we have  an  isomorphism of algebras
\begin{align*}
 U(\fing)/\mc{J}_W\isomap U(\fing)/\bra \mc{J}_W,\Omega\ket
 \times
U(\fing)/\bra \mc{J}_W,\Omega-c_0\ket.
\end{align*}
As we have explained above,
 $\bra \mc{J}_W,\Omega-c_0\ket=\mc{J}_0$.
Also,
since
$\mc{J}_W$ contains $(\Omega-c_0)\fing$,
$\bra \mc{J}_W,\Omega\ket$ contains $\fing$.
Therefore
 $U(\fing)/\bra \mc{J}_W,\Omega\ket=\C$ as required.
  \end{proof}

\section{A lifting of Joseph ideals} \label{sec:Joseph}
For a $\Z_{\geq 0}$-graded vertex algebra $V=\bigoplus_{d}V_d$, 
let $A(V)$ be Zhu's algebra of $V$:
\begin{align*}
 A(V)=V/V\circ V,
\end{align*}
where $V\circ V$ is the $\C$-span of the vectors
\begin{align*}
a\circ b:=\sum_{i\geq 0}\begin{pmatrix}
			 \Delta\\ i
			\end{pmatrix}a_{(i-2)}b
\end{align*}for $a\in V_{\Delta}$, $\Delta\in \Z_{\geq 0}$, $b \in V$,
and
$V\ra (\End V)[[z,z^{-1}]]$,  $a\mapsto \sum_{n\in \Z}a_{(n)}z^{-n-1}$,
denotes the state-field correspondence.
The space $A(V)$ is a unital associative algebra
with respect to the  multiplication defined by
\begin{align*}
 a * b:=\sum_{i\geq 0}\begin{pmatrix}
			 \Delta\\ i
			\end{pmatrix}a_{(i-1)}b
\end{align*}
for $a\in V_{\Delta}$, $\Delta\in \Z_{\geq 0}$, $b \in V$.
More generally, for a $V$-module $M$, 
a bimodule $A(M)$ over $A(V)$ is defined similarly
(\cite{FreZhu92}).

Zhu's algebra $A(V)$ naturally acts on the top degree component $M_{top}$ of a
$\Z_{\geq 0}$-graded
$V$-module $M$,
and
$M\mapsto M_{top}$ gives  \cite{Zhu96}
a one-to-one correspondence between simple graded
  $V$-modules
and simple $A(V)$-modules.

The vertex algebra $V$ is called a {\em chiralization} of an algebra $A$ if
$A(V)\cong A$.

For instance, 
consider the universal affine vertex algebra $V^k(\fing)$.
A $V^k(\fing)$-module is the same as a smooth $\affg'$-module of level
$k$,
where $\affg'=[\affg,\affg]=\fing[t,t^{-1}]\+ \C K$.
Zhu's algebra
 $A(V^k(\fing))$ is naturally isomorphic to $U(\fing)$ (\cite{FreZhu92},
 see also \cite[Lemma 2.3]{A12-2}),
and hence, 
$V^k(\fing)$ is a chiralization of $U(\fing)$.
The top degree component 
of the irreducible highest weight representation $L(\lam)$ of $\affg$
with highest weight $\lam$ is $L_{\fing}(\bar \lam)$,
where $\bar \lam $ is the 
restriction of $\lam$ to the
  Cartan subalgebra of $\fing$.

Let $\widehat{\mc{J}}_k
$ be the unique maximal ideal of $V^k(\fing)$, so that 
\begin{align*}
 V_k(\fing)=V^k(\fing)/\widehat{\mc{J}}_k.
\end{align*}
We have the exact sequence
$A(\widehat{\mc{J}}_k)\ra U(\fing)\ra A(V_k(\fing))\ra 0$ since the functor $A(?)$ is
right exact and thus
$A(V_k(\fing))$ is the quotient of $U(\fing)$ by the the image
$\mc{I}_k$
of $A(\widehat{\mc{J}}_k)$ in $U(\fing)$:
\begin{align*}
 A(V_k(\fing))=U(\fing)/\mc{I}_k.
\end{align*}

One may ask whether 
 $\mc{I}_k$
coincides with the Joseph ideal $\mc{J}_0$ for some $k\in \C$,
so that
$V_k(\fing)$ is a chiralization of $U(\fing)/\mc{J}_0$.
But  this can never happen.
Indeed,
$U(\fing)/\mc{J}_0$ does not admit  finite dimensional representations
while
$\C$  is always an $A(V_k(\fing))$-module  as $V_k(\fing)$ is a module over
 itself
and $V_k(\fing)_{top}=\C$.
However, by Proposition~\ref{Pro:enlarged-Joseph},
it makes sense to ask the same question for the ideal $\mc{J}_W$.

%Recall that the series
%\begin{align*}
% A_1\subset A_2\subset G_2\subset D_4\subset F_4\subset E_6\subset
% E_7\subset E_8
%\end{align*}
%is called the 
%{\em Deligne  exceptional series}.

 \begin{Th} \label{Th:lifting-of-Joseph-ideal}
Assume that $\fing$ belongs to the Deligne 
exceptional series outside the type $A$
and that 
$k=-h^{\vee}/6-1$. 
%\item $\fing=\mf{sp}_{2l}$, $l\geq 2$, and $k=-1/2$.
Then $V_k(\fing)$ is a chiralization of 
$U(\fing)/\mc{J_W}$, that is,
\begin{align*}
A(V_k(\fing))\cong U(\fing)/\mc{J_W} \cong \C\times U(\fing)/\mc{J}_0.
\end{align*}
In particular, since $\mc{J}_0$ is maximal,
      the irreducible highest weight representation $L(\lam)$ of
 $\affg$ is a $V_k(\fing)$-module if and only if 
\begin{align*}
\bar \lam=0\quad\text{or}\quad
\on{Ann}_{U(\fing)}L_{\fing}(\bar \lam)=\mc{J}_0.
\end{align*} \end{Th}

According to \cite[4.3]{Jos98}, the weights $\mu$ 
such that ${\rm Ann}_{U(\fing)}L_{\fing}(\mu) = \mc{J}_0$ 
are 
\begin{align*}
 w\circ (\lam_0-\rho) := w(\lam_0)-\rho,\quad w\in {\rm W}_{0},
\end{align*}
where  
the weight $\lam_0$ and the subset ${\rm W}_{0}$ of the Weyl group 
${\rm W}$ of $\fing$
are described in Table \ref{tab:weighs}. 
Here
%For the next theorem, and its proof, 
we adopt the standard 
Bourbaki numbering for the simple roots 
$\{\alpha_1,\ldots,\alpha_1\}$ of $\fing$, 
%denote by $\Delta$ the corresponding root system 
and we denote by $\varpi_1,\ldots,\varpi_l$ the 
corresponding fundamental weights. 

%Here, as usual, the action $\circ$ is given by 
%\begin{align*} 
%w\circ \lambda = w(\lambda+\rho) -\rho , \qquad w \in W , \quad  \lam \in \mf{h}^* .
%\end{align*}

 {\begin{table}[h]\small
 \begin{center}
\begin{tabular}{|c|c|c|c|}
\hline  & $-\frac{h^{\vee}}{6}-1$ &$\lam_0$ & ${\rm W}_{0}$\\[0.25em]
\hline && & \\[-1em]
$G_2$ & $-\frac{5}{3}$ & $\varpi_1+\frac{1}{3}\varpi_2$ & $\{1,s_2\}$ \\[0.25em]
\hline && & \\[-1em]
$D_4$ & $-2$ & $\varpi_1+\varpi_3+\varpi_4$ &$\{1, s_1, s_3, s_4\}$\\[0.25em]
\hline && & \\[-1em]
$F_4$ & $-\frac{5}{2}$ &
	 $\frac{1}{2}\varpi_1+\frac{1}{2}\varpi_2+\varpi_3+\varpi_4$
& $\{1,s_1,s_2\}$
\\[0.25em]
\hline && & \\[-1em]
$E_6$ & $-3$ & $\varpi_1+\varpi_2+\varpi_3+\varpi_5+\varpi_6$
&$\{1,s_2,s_3,s_1s_3,s_5,s_6s_5\}$\\[0.25em]
\hline && & \\[-1em]
$E_7$ & $-4$ & $\varpi_1+\varpi_2+\varpi_3+\varpi_5+\varpi_6+\varpi_7$
&$\{1,s_2,s_3,s_1s_3,s_5,s_6s_5,s_7s_6s_5\}$ \\[0.25em]
\hline && & \\[-1em]
$E_8$ & $-6$ & $\varpi_1+\varpi_2+\varpi_3+\varpi_5+\varpi_6+\varpi_7+\varpi_8$
&$\{1,s_2,s_3,s_1s_3,s_5,s_6s_5,s_7s_6s_5,s_8s_7s_6s_5\}$\\[0.25em]
\hline 
\end{tabular}
 \end{center}
\vspace{.25cm}
\caption{$-h^{\vee}/6-1$, $\lam_0$ and ${\rm W}_{0}$} \label{tab:weighs}
 \end{table}}
Note that the last statement of 
 Theorem \ref{Th:lifting-of-Joseph-ideal}
reproves the earlier results
 \cite[Proposition 3.6 (1)]{AxtLee11} for type $G_2$,
 \cite[Theorem 4.3]{Per13} for type $D_4$
 and \cite[Theorem 6.4]{Per07typeF} for type $F_4$.

For 
types $G_2$ and $F_4$, the level $k=-h^{\vee}/6-1$ is {\em admissible},
that is,
$k\Lam_0$ is an admissible weight \cite{KacWak89} for $\affg$.
Using  \cite[Proposition 3.3]{A12-2}
one finds that 
\begin{align*}
\{k\Lam_0, w\circ (\lam_0-\rho)+k\Lam_0\mid w\in {\rm W}_{0}\}
\end{align*}
is exactly the set of admissible weights of level $k$
whose integral Weyl group is isomorphic to that of $k\Lam_0$,
which  agrees with    \cite[Main Theorem]{A12-2}.

Theorem \ref{Th:lifting-of-Joseph-ideal} will be proven  
at the end of Section \ref{sec:sing_vectors}. 
 
\section{Singular vectors of affine vertex algebra of degree $2$} \label{sec:sing_vectors}
By the PBW theorem, we have
$V^k(\fing)\cong U(\fing[t^{-1}]t^{-1})$  as $\C$-vector spaces.
Below we often identify $V^k(\fing)$ with $U(\fing[t^{-1}]t^{-1})$.

The vertex algebra $V^k(\fing)$ is naturally graded:
\begin{align*}
V^k(\fing)
 =\bigoplus_{d\in \Z_{\geq 0}}V^k(\fing)_d,\quad
V^k(\fing)_d=\{v \in V^k(\fing)\mid Dv=-d v\}.
\end{align*}
%\marginpar{\tiny This grading corresponds to the natural 
%grading on $V^k(\fing)$ as vertex algebra: why do we need of 
%$D$? I feel every think works with usual affine KM algebra, no?
%Yes,can you change it accordingly?}
Note that each homogeneous component $V^k(\fing)_d$ is a
finite-dimensional $\fing$-submodule of $V^k(\fing)$.
 \begin{Lem}We have a $\fing$-module embedding
\begin{align*}
 \sigma_d :S^d(\fing)\hookrightarrow V^k(\fing)_d, 
\quad x_1\dots x_d\mapsto \frac{1}{d!}\sum_{\sigma\in
 \mathfrak{S}_d}x_{\sigma(1)}(-1)\dots x_{\sigma(d)}(-1).
\end{align*}
%\marginpar{\tiny Finally, I removed ${\bf 1}$ since we identified with 
%$U(\fing[t^{-1}]t^{-1}$. So, I have changed a little the proof.}
 \end{Lem}
Let $v$ be a singular vector in $S^d(\fing)$.
Then
$\sigma_d(v)$  is a singular vector of $V^k(\fing)$ if and only if 
$f_{\theta}(1)\sigma_d(v)=0$. 
For $d=2$, we will simply denote by $\sigma$ the embedding $\sigma_d$. 

Let $W=\bigoplus_{i}W_i$ be the decomposition of $W$
into irreducible submodules,
and let $w_i$ be a highest weight vector of $W_i$.

 \begin{Th}\label{Th2}
 %\marginpar{\tiny{Maybe it is better to write the
 %statement like Theorem 1.1. What do you think?}}
\begin{enumerate}

\item Assume that $\fing$ belongs to the Deligne exceptional series 
outside the  type $A$. 

\begin{enumerate}

\item For any $i$, $\sigma(w_i)$ is a singular vector of $V^k(\fing)$ if and only if
	 	\begin{align*}
k=-h^{\vee}/6-1.
		\end{align*} 
\item  %Let $\fing$ be of type  $D_4$, $E_6$,
	 %$E_7$, $E_8$. 
	 Assume that $\fing$ is not of type $G_2$. 
For each $n\in \Z_{\geq 0}$ and each $i$,
  $\sigma(w_i)^{n+1} $ is a singular vector of
	 $V^k(\fing)$ if and only if
	 	\begin{align*}
k=n-h^{\vee}/6-1.
		\end{align*}
\end{enumerate}

%\item  Let $\fing$ be of type $G_2$ or $F_4$, so that
%$W=W_1\cong L_{\fing}(\theta+\theta_1)=L_{\fing}(2\varpi_1)$.
%Then the vector $\sigma(w_1)$ is a singular vector of $V^k(\fing)$ if
%      and only if
%      \begin{align*} k=-h^{\vee}/6-1.  
%      %=-5/2.
%      \end{align*}

%\item {\rm (\cite{AxtLee11})} Let $\fing$ be of type $G_2$, so that
%$W=W_1\cong L_{\fing}(\theta+\theta_1)=L_{\fing}(2\varpi_1)$.
%Then the vector $\sigma(w_1)$ is the singular vector of $V^k(\fing)$ if
%      and only if 
%      \begin{align*}k=-5/3.
%       \end{align*} 
\item Let $\fing$ be of type $B_l$, $l\geq 3$, so that
	$W=W_1\+ W_2$ where
	$W_1\cong L_{\fing}(\theta+\theta_1)=L_{\fing}(2\varpi_1)$ and $W_2\cong 
	L_{\fing}(\theta+\theta_2)=L_{\fing}(\varpi_4)$ if $l \geq 5$ 
	(and $W_2\cong 
	L_{\fing}(\theta+\theta_2)=L_{\fing}(2\varpi_l)$ if $l=3,4$).
\begin{enumerate}
 \item{\rm (\cite{Per07})}
For each $n\in \Z_{\geq 0}$,
 $\sigma(w_1)^{n+1}$ is a singular vector of
	$V^k(\fing)$ if and only if 
\begin{align*}
k=n-l+3/2.
\end{align*}
\item
For each $n\in \Z_{\geq 0}$,
	$\sigma(w_2)^{n+1}$ is a
	singular vector of
	$V^k(\fing)$ if and only if 
\begin{align*}
k=n-2.
\end{align*}

\end{enumerate}

\item {\rm (\cite{Ada94})}
Let $\fing$ be of type $C_l$, $l\geq 2$,
so that $W=W_1 \oplus W_2$ where $W_1 \cong L_{\fing}(\theta+\theta_1)=L_{\fing}(2\varpi_2)$ 
and $W_2 \cong L_{\fing}(\frac{1}{2}\theta+\theta_1)=L_{\fing}(\varpi_2)$.
For each $n\in \Z_{\geq 0}$,
$\sigma(w_1)^{n+1}$ is a singular vector of $V^k(\fing)$ 
if and only if 
\begin{align*}
 k=n-1/2.
\end{align*}

  \item Let $\fing$ of type $D_l$, $l\geq 5$, so that
	$W=W_1\+ W_2$ where
	$W_1\cong L_{\fing}(\theta+\theta_1)=L_{\fing}(2\varpi_1)$ and 
	$W_2\cong L_{\fing}(\theta+\theta_2)=L_{\fing}(\varpi_4)$
	%\color{red}
	if $l\geq 6$ 
	(and $W_2\cong L_{\fing}(\theta+\theta_2)=L_{\fing}(\varpi_4+\varpi_5)$ if $l =5$). 
	%\color{black}
	%\marginpar{\tiny \color{red}I checked the other types and the weights are correct.}
	%$W_2\cong L_{\fing}(\varpi_1+\varpi_2+\varpi_3+\varpi_n)$.
\begin{enumerate}
 \item {\rm (\cite{Per13})}
For each $n\in \Z_{\geq 0}$,
 $\sigma(w_1)^{n+1}$ is a singular vector of
	$V^k(\fing)$ if and only if 
\begin{align*}
k=n-l+2.
\end{align*}
\item  For each $n\in \Z_{\geq 0}$,
	$\sigma(w_2)^{n+1}$ is a
	singular vector of
	$V^k(\fing)$ if and only if 
\begin{align*}
k=n-2.
\end{align*}
\end{enumerate}
      
\end{enumerate}
\end{Th}

Note that (1) for $D_4$ is also a particular case of \cite{Per13}, 
that (1) (a) for $G_2$ was  proved in \cite{AxtLee11}, 
and that (1) for $F_4$ was proved in \cite{Per07typeF}. 
%In addition, note that (4) actually holds for $l=4$. 

\begin{proof}
(1)  Assume that $\fing$ is of type  $D_4$, $E_6$, $E_7$, $E_8$. 
Then it is enough to prove (b). 

For $E_6$, $E_7$, $E_8$, $W=W_1$.  
For $D_4$, $W=W_1\oplus W_2 \oplus W_3$. 
Using the Dynkin automorphism, we can assume that $i=1$, and that 
$W_1=L_{\fing}(2\varpi_1)$. 

 %\color{red}
 For types $E_6$ and $E_7$, 
$\fing$ is of {\em depth one}, \cite[Chapter IV, Definition 1]{Gar}, 
and $(\theta-\theta_1)/2$ is not a root. 
%\color{black}

Then we apply \cite[Chapter IV, Proposition 11]{Gar} to construct a singular vector 
$w_1$ for $W_1$. Table \ref{tab:pairs} describes the pairs of positives 
roots $(\beta_j,\delta_j)$ such that 
$$\beta_j +\delta_j= \theta - \theta_1.$$ 
The number of such pairs turns out to be 
equal to $h^\vee/6+1$. 
In this table, a positive root $\gamma$ 
is represented by $(k_1,\ldots,k_l)$ if $\gamma=\sum_{j=1}^l k_j \alpha_j$. 

Choose a Chevalley basis 
$\{h_i\}_i\cup \{e_\alpha, \, f_\alpha\}_{\alpha}$ of $\fing$ 
so that the conditions of \cite[Chapter IV, Definition 6]{Gar} 
are fulfilled, that is 
\begin{eqnarray} \label{eq:brack}
&& \forall \, j, \quad [e_{\delta_j},[e_{\beta_j},e_{\theta_1}]] = e_{\theta}, \quad 
[e_{\beta_j},e_{\theta_1}] = e_{\beta_j+\theta_1}, \quad 
[e_{\delta_j},e_{\theta_1}] = e_{\delta_j+\theta_1}.
\end{eqnarray}
Then set 
$$w_1:=e_{\theta}e_{\theta_1} - \sum_{k=1}^{\frac{h^\vee}{6}+1} 
e_{\beta_j+\theta_1} e_{\delta_j +\theta_1},$$ 
so that 
\begin{eqnarray*}
\sigma(w_1)&= &\frac{1}{2}(e_{\theta}(-1)e_{\theta_1}(-1)+e_{\theta_1}(-1)e_{\theta}(-1) \\ 
&& \quad - \sum_{k=1}^{\frac{h^\vee}{6}+1} (e_{\beta_j+\theta_1} (-1) e_{\delta_j +\theta_1}(-1) 
+ e_{\delta_j +\theta_1}(-1)  e_{\beta_j+\theta_1} (-1) )). %\;  \in  \; U(\affg)
\end{eqnarray*}
%so that  $$\sigma_2(w_1) = v_1 {\bf 1}.$$ 
We observe using the relations (\ref{eq:brack}) that for each $j$, 
\begin{eqnarray} \label{eq2:brack}
 [[f_\theta,e_{\beta_j+\theta_1} ],e_{\delta_j+\theta_1}]
=  [[f_\theta, e_{\delta_j+\theta_1}] ,e_{\beta_j+\theta_1}] = -e_{\theta_1}.
\end{eqnarray}
By (\ref{eq2:brack}), we get:
\begin{eqnarray*}
f_{\theta}(1).\sigma(w_1) &=&  ([f_\theta,e_\theta](0) + k + \frac{h^\vee}{6}+1)e_{\theta_1}(-1) \\
&& -  \sum_{k=1}^{\frac{h^\vee}{6}+1} (e_{\beta_j+\theta_1} (-1) [f_\theta,e_{\delta_j +\theta_1} ](0) 
+ e_{\delta_j +\theta_1}(-1)  [f_\theta,e_{\beta_j+\theta_1}] (0) ).
\end{eqnarray*}
Observe that 
$$[f_\theta,e_\theta](0).\sigma(w_1) = - 2 \sigma(w_1)$$ 
since 
$\langle \theta+\theta_1, \theta^\vee \rangle= \langle \theta, \theta^\vee \rangle= 2,$ 
and that 
$$ [f_\theta,e_{\delta_j +\theta_1} ](0).\sigma(w_1) 
=  [f_\theta,e_{\beta_j+\theta_1}] (0). \sigma(w_1) =0$$
since $-\theta+\delta_j +\theta_1$, $-\theta+\beta_j +\theta_1$ are perpendicular to 
$\theta+\theta_1$, the weight of $\sigma(w_1)$, for each $j$. 
In addition, since $\beta_j+2\theta_1$, $\delta_j+2\theta_1$ are not roots, 
$[e_{\theta_1}(-1),\sigma(w_1)]=0$. 
So, for any $n \in \Z_{\geq 0}$ we get, 
%\marginpar{\tiny{I changed the  line. OK?}}
%\marginpar{\tiny{$v_1$ is not defined}}
%\marginpar{\tiny{I think we also need $[e_{\theta_1}(-1),v_1]=0$,
%which follows from
%$\beta_j+2\theta_1$, $\delta_j+2\theta_1$ are not roots.}}
\begin{align*} 
&f_\theta(1).\sigma(w_1)^{n+1} \\
&= 
\sigma(w_1)^n  (k + \frac{h^\vee}{6}+1) e_{\theta_1}(-1) 
 + \sum_{j=1}^{n} (
\sigma(w_1)^{n-j}  ([f_\theta,e_\theta](0) + k + \frac{h^\vee}{6}+1) . \sigma(w_1)^{j} e_{\theta_1}(-1))
\\
&=\sum_{j=0}^{n} (- 2j + k + \frac{h^\vee}{6}+1 ) \sigma(w_1)^{n} e_{\theta_1}(-1) 
= (n+1) (- n  + k + \frac{h^\vee}{6}+1) \sigma(w_1)^ne_{\theta_1}(-1).
\end{align*}
Hence $\sigma(w_1)^{n+1}$ is a singular vector of $V^k(\fing)$ 
for $k= n -h^\vee/6 -1$. 

%\marginpar{\tiny\color{red} I write apart the case $E_8$. Is it ok?}
 %\color{red}
 Assume that $\fing$ has type $E_8$. 
Then $\fing$ is not of depth one 
and we follow the construction of \cite[Chapter IV, \S 4]{Gar}. 
According to \cite[Chapter IV, \S 4]{Gar}, there is a positive root $\alpha$ 
such that the algebra $\tilde{\fing}$ generated by $e_{\alpha},e_{2}, 
\ldots,e_{8},f_{\alpha},f_2,\ldots,f_8$ has type $D_8$, 
where $e_i,f_i,i=1,\ldots,8$ are the generators of a Chevalley basis of $\fing$ 
corresponding to the simple roots $\alpha_1,\ldots,\alpha_8$ in the Bourbaki numbering.  
Moreover, we have that $\alpha=\theta_1$. 
Then we apply the construction of \cite[Chapter IV, \S 1]{Gar} 
to the algebra $\tilde{\fing}$ which is of depth one. 
One can choose our Chevalley basis 
$\{h_i\}_i\cup \{e_\alpha, \, f_\alpha\}_{\alpha}$ of $\fing$ 
so that the conditions of \cite[Chapter IV, Definition 6]{Gar} 
are fulfilled for $\tilde{\fing}$. 
Note that the highest root of $\tilde{\fing}$ is $\theta$, that is, 
the same as for $\fing$. 

Then we apply as in cases $E_6,E_7$ the construction of 
\cite[Chapter IV, Proposition 11]{Gar}. 
Table \ref{tab:pairs} describes the pairs of positives 
roots $(\beta_j,\delta_j)$ such that 
$$\beta_j +\delta_j= \theta - \theta_1.$$ 
The number of such pairs is $h^\vee/6+1$ too. 

Then we set 
$$w_1:=e_{\theta}e_{\theta_1} - \sum_{k=1}^{\frac{h^\vee}{6}+1} 
e_{\beta_j+\theta_1} e_{\delta_j +\theta_1}.$$
We verify as for the types $E_6,E_7$ that 
$\sigma(w_1)^{n+1}$ is a singular vector of $V^k(\fing)$ 
for $k= n -h^\vee/6 -1$. 
%\color{black}
{\begin{table}[h] \tiny
\begin{center}
\begin{tabular}{|c|c|c|c|c|}
\hline &&& & \\[-0.5em] 
Type & $D_4$ & $E_6$  & $E_7$ & $E_8$  \\[0.25em] 
\hline &&& & \\[-0.5em] 
${h^\vee}/{6}+1$ & 2 & 3 & 4 & 6 \\[0.25em] 
\hline &&& & \\[-0.25em] 
$\theta$ & $(1211)$ & $(122321)$ & $(2234321)$ & $(23465432)$  \\[0.25em] 
\hline &&& & \\[-0.25em] 
$\theta_1$ & $(1000)$ & $(101111)$ & $(0112221)$ 
& $(22343210)$  \\[0.25em] 
\hline &&& & \\[-0.25em] 
$(\beta_j,\delta_j)$, 
& $ (0100) , (0111) $ & 
$(010 000),(011 210)$  &  $(100 000 0),(112 210 0)$ & $(000 000 01),(011 222 21)$ \\[0.25em]
$\beta_j+\delta_j=\theta- \theta_1$ & $(0101),(0110)$ &  
$(010 100),(011 110)$  &  $(101 000 0),(111 210 0)$ & $(000 000 11),(011 222 11)$ \\[0.25em] 
 & & 
$(010 110),(010 100)$ &  $(101 100 0),(111 110 0)$ & $(000 001 11),(011 221 11)$  \\[0.25em]  
&&&$(101 110 0),(111 100 0)$  & $(000 011 11),(011 211 11)$ \\[0.25em] 
&&& &  $(000 111 11),(011 111 11)$  \\[0.25em] 
&&&& $(010 111 11),(001 111 11)$  \\[0.25em] 
\hline
\end{tabular}
\vspace{.25cm}
\caption{Data for $D_4$, $E_6$, $E_7$, $E_8$% and $F_4$
} \label{tab:pairs}
\end{center}
\end{table}}

\smallskip

(2) (b) and (4) (b) 
Assume that $\fing$ is of type $B_l$, $l\geq 3$, or of type $D_l$, $l\geq 5$. 
Then in both cases, $\theta_2$ is the highest root of the root system generated 
by $\alpha_3,\ldots,\alpha_l$, $(\theta-\theta_2)/2$ is not a root and 
there are precisely two pairs $(\beta_j,\delta_j)$ 
such that $\beta_j +\delta_j=\theta-\theta_2$. 
Namely, these pairs are:
%\marginpar{\tiny{something wrong. Yes!}}
$$(\beta_1,\delta_1)=(\alpha_2,\alpha_1+\alpha_2+\alpha_3) 
\quad \text{ and }\quad (\beta_2,\delta_2)=(\alpha_2+\alpha_3 ,\alpha_1+\alpha_2).$$
According to \cite[Chapter IV,Proposition 11]{Gar}, 
$$w_2:=e_{\theta}e_{\theta_2} - \sum_{k=1}^{2} 
e_{\beta_j+\theta_2} e_{\delta_j +\theta_2}$$ 
is a singular vector for $\fing$. 
Moreover, all bracket relations $(\ref{eq:brack})$ and $(\ref{eq2:brack})$ 
hold as in case~(1)\footnote{For $B_3$, a factor $2$ appears in some brackets 
but this does not affect the final result.}, with $\theta_2$ in place of $\theta_1$. 
Hence we get, 
%\marginpar{\tiny{OK?}}
$$f_{\theta}(1).\sigma(w_2)^{n+1} = (-n+k+2)\sigma(w_2)^n e_{\theta_2}(-1).$$
The statement follows. 
\end{proof}

\begin{Rem} 
If $\fing$ is of type $C_l$, $l \geq 3$, 
we can construct a singular vector for $V^k(\fing)$ of weight $\frac{1}{2}(\theta+\theta_1)$ 
with $k=-(l+6)/2$ as follows. 

Set 
$$\theta_0 := (\theta+\theta_1)/2=\alpha_1+2(\alpha_2+\cdots+\alpha_{l-1}) +\alpha_l.$$
For $j \in \{2,\ldots,l\}$, set 
\begin{eqnarray*} 
\beta_j := \alpha_1 + \alpha_2 + \cdots + \alpha_{j-1}, && 
\delta_j :=  \alpha_2 + \cdots + \alpha_{j-1} +2(\alpha_j+ \cdots 
+\alpha_{l-1}) + \alpha_l. 
\end{eqnarray*}
For $j \in \{3,\ldots,l\}$, set 
\begin{eqnarray*} 
&&\beta'_j := \alpha_2 +  \cdots + \alpha_{j-1}, \qquad 
\delta'_j :=  \alpha_1 + \cdots + \alpha_{j-1} +2(\alpha_j+ \cdots 
+\alpha_{l-1}) + \alpha_l. 
\end{eqnarray*}
Then 
$$\forall \, j \in \{3,\ldots,l\},\quad \beta_j + \delta_j =\beta'_j + \delta'_j 
= \theta_0 = \frac{1}{2}(\theta + \theta_1)
\quad\text{ and }\quad \beta_2 +\delta_2 = \theta_0.$$
We can choose a Chevalley basis of $\fing$ such that the vector 
\begin{eqnarray*}   
v_2 &:= & 
e_\theta (-1) e_{ - \alpha_1}(-1)    -\frac{1}{2}   h_1 (-1) e_{\theta_0}(-1)
+ e_{\theta_0}(-2)   \\
&& \; - e_{\beta_2}(-1) e_{\delta_2}(-1) -\frac{1}{2} \sum_{j=3}^l (  e_{\beta_j}(-1) e_{\delta_j}(-1) 
- e_{\beta'_j}(-1) e_{\delta'_j}(-1))  
\end{eqnarray*} 
is singular for $V^k(\fing)$ with $k=-(l+6)/2$. 
The verifications are left to the reader. 
This remark will be not used in the sequel. 
\end{Rem}

 \begin{proof}[Proof of Theorem \ref{Th:lifting-of-Joseph-ideal}] 
Let $\fing$, $k$ be as in Theorem.
Then $\sigma(w_i)$ is a singular vector of $V^k(\fing)$ for all $i$ by
  Theorem \ref{Th2}.
Let
$N$ be the submodule of $V^k(\fing)$ generated by $\sigma(w_i)$ for all
  $i$,
and set $\tilde{V}_k(\fing)=V^k(\fing)/N$.
By construction the image of $A(N)$ in $U(\fing)$ is $\mc{J}_W$.
Hence
\begin{align*}
 A(\tilde{V}_k(\fing))=U(\fing)/\mc{J}_W.
\end{align*}
It remains to show that $\tilde{V}_k(\fing)=V_k(\fing)$,
that is, $\tilde{V}_k(\fing)$ is simple.
(In the case that $k$ is admissible, that is, if $\fing$ is of type 
$G_2$,  $F_4$, this follows from \cite{KacWak88}.
Also,
this
 has been proved in \cite{Per13}
in the case that $\fing$ is of type $D_4$.)
 
Suppose that $\tilde{V}_k(\fing)$ is not simple,
or equivalently,
$\tilde{V}_k(\fing)$ is reducible as a
  $\affg$-module.
Then there is at least one non-zero weight singular
vector, say, $v$.
Let $\mu$ be the weight of $v$,
and 
let $M$ be a submodule of $\tilde{V}_k(\fing)$ generated by $v$.
Since $M_{top}
=L_{\fing}(\bar \mu)$,
$L_{\fing}(\bar \mu)$ is a module over
  $A(\tilde{V}_k(\fing))=U(\fing)/\mc{J}_W
=\C\+U(\fing)/\mc{J}_0$.
On the other hand  $L_{\fing}(\bar \mu)$
 is finite-dimensional
since it is a submodule of
  $V^k(\fing)_d$ for some $d$.
This implies that  $L_{\fing}(\bar \mu)$ cannot be a
  $U(\fing)/\mc{J}_0$-module.
Therefore $\bar \mu=0$.
This implies that
$v$ coincides with the highest weight vector of 
$\tilde{V}_k(\fing)$ up to nonzero multiplication,
which is a 
contradiction.
 \end{proof}
 \section{Proof of Theorem \ref{Th1}}\label{sec:proof-of-th1}
\label{sec:proofs}
Let  $\fing$ be of type $D_l$, $l\geq 4$, $E_6$, $E_7$, or $E_8$.

For $n\in \Z_{\geq 0}$, set
\begin{align}
 k_n=\begin{cases}
n-h^{\vee}/6-1      &\text{if $\fing$ is of type $D_4$,  $E_6$,
      $E_7$, $E_8$},\\
n-2&\text{if $\fing$ is of type $D_l$, $l\geq 5$}.
     \end{cases}
\label{eq:value-of-k}
\end{align}
 Let $N$ be the submodule
 of
 $V^k(\fing)$ generated by $\sigma(w_i)^{n+1}$ for all $i$
 for type $D_4$,  $E_6$, $E_7$, $E_8$,
 and by $\sigma(w_1)^{n+l-3}$ and $\sigma(w_2)^{n+1}$
 for type $D_l$, $l\geq 5$, and let
 \begin{align*}
  \tilde V_{k_n}(\fing):=V^{k_n}(\fing)/N.
 \end{align*}
 \begin{Conj}\label{Conj:simplicity}
$\tilde{V}_{k_n}(\fing)=V_{k_n}(\fing)$, that is,
$\tilde{V}_{k_n}(\fing)$ is simple, if $k_n<0$.
 \end{Conj}
We have proven Conjecture  \ref{Conj:simplicity}
in the case that $n=0$
in type $D_4$, $E_6$, $E_7$, $E_8$
in the proof of Theorem \ref{Th:lifting-of-Joseph-ideal}.

 \begin{Rem} 
If $k_n\geq 0$, 
$\tilde{V}_{k_n}(\fing)$ is obviously not simple 
as the maximal
  submodule of $V^{k_n}(\fing)$ is generated by $e_{\theta}(-1)^{k_n+1}$.
 \end{Rem}
 \begin{Pro}\label{Pro:variety-intermidiate}
For each $n\geq 0$,
we have
$X_{\tilde{V}_{k_n}(\fing)}=\overline{\mathbb{O}_{min}}$.
 \end{Pro}
  \begin{proof}
Set $k=k_n$.
The exact sequence $0\ra N\ra V^k(\fing)\ra \tilde{V}_k(\fing)\ra 0$
induces an exact sequence
\begin{align*}
 N/\fing[t^{-1}]t^{-2}N
\ra V^k(\fing)/\fing[t^{-1}]t^{-2}V^k(\fing)
\ra
 \tilde{V}_k(\fing)/\fing[t^{-1}]t^{-2}\tilde{V}_k(\fing)\ra 0.
\end{align*}
 Under the isomorphism
$V^k(\fing)/\fing[t^{-1}]t^{-2}V^k(\fing)\cong S(\fing)$,
the image of $ N/\fing[t^{-1}]t^{-2}N$ in
  $V^k(\fing)/\fing[t^{-1}]t^{-2}V^k(\fing)$ is identified with 
  the ideal $J$ of $S(\fing)$ generated by
some powers of $w_i$ for all $i$.
  Hence $J\subset J_W\subset    \sqrt{J}$.
   Therefore,
\begin{align*}
\sqrt{J}=\sqrt{J_W}=J_0
\end{align*}
by Lemma \ref{Lem:Gan-Savin}
 as required.  \end{proof}
 \begin{proof}[Proof of Theorem \ref{Th1}]
For $\fing$ of type $A_1$, $A_2$, $G_2$, $F_4$,
the number $-h^{\vee}/6-1$  is admissible, 
and the statement (1) of the theorem is  a special case of 
 \cite[Theorem 5.14]{Ara09b}.
So let us assume that $\fing$ is of type $D_l$, $l\geq 4$, $E_6$, $E_7$,
  or $E_8$ as above.
Since $V_{k_n}(\fing)$ is a quotient of $\tilde{V}_{k_n}(\fing)$,
Proposition \ref{Pro:variety-intermidiate}
implies that
\begin{align*}
 X_{V_{k_n}(\fing)} \subset \overline{\mathbb{O}_{min}}={\mathbb{O}_{min}}\cup \{0\}.
\end{align*}
Therefore  $X_{V_{k_{n}}(\fing)}$ is either $\{0\}$ or
  $\overline{\mathbb{O}_{min}}$.
The assertion follows since $X_{V_k(\fing)}=\{0\}$ if and only if $k\in
  \Z_{\geq 0}$ by \cite[Proposition 4.25]{Ara09b} (see also Theorem
  \ref{Th:previous} (2) and (3) (a)).
\end{proof}
The following assertion was proved in \cite{Per13}
in the case that $\fing$ is of type $D_4$ and 
$k=-2$.
 \begin{Co} \label{cor:finitely-modules}
Let $\fing$, $k$ be as in Theorem \ref{Th1}.
Then $V_k(\fing)$ has only finitely many simple modules in the category $\mathcal{O}$.
 \end{Co}
 \begin{proof}
%Let $A(V)$ be Zhu's algebra of a vertex algebra $V$.
%Note since $V_k(\fing)$ is a quotient of $V^k(\fing)$,
%$A(V_k(\fing))$ is a quotient of $A(V^k(\fing))=U(\fing)$.
%
%Recall that $R_{V_k(\fing)} = V_k(\fing) /\fing[t^{-1}]t^{-2}  V_k(\fing)$. 
By \cite[Proposition 2.17(c)]{De-Kac06}, \cite[Proposition 3.3]{ALY} there is a surjection
\begin{align*}
 R_{V_k(\fing)}\twoheadrightarrow \gr A(V_k(\fing))
\end{align*}
of Poisson algebras,
where $\gr A(V_k(\fing))$ is the associated graded algebra of
  $A(V^k(\fing))$ with respect to Zhu's filtration \cite{Zhu96}.
Hence $\on{Specm}(\gr A(V_k(\fing)))\subset X_{V_k(\fing)}\subset
  \mc{N}$.
It follows that the center $Z(\fing)$ of $U(\fing)$ acts finitely 
on $A(V_k(\fing))$ and 
therefore
 there are only finitely many possible central characters of simple
  modules
of $A(V_k(\fing))$.
 \end{proof}

 \begin{Rem}
Let $\fing,f$ be as in Theorem \ref{Th1}.
As in the same way as \cite[Theorem 9.5]{A2012Dec},
one finds that
$X_{V_k(\fing)}=\on{Specm} (\gr A(V_k(\fing)))$,
which
 gives another evidence for \cite[Conjecture 1]{A2012Dec}.
 \end{Rem}

%\marginpar{\tiny If $k$ is admissible $X_{V_{k}(\fing)} \subset \mc{N}$, but it 
%is not $\overline{\mathbb{O}_{min}}$ in general?}
  \begin{Conj}\label{Conj:when-variety-is-the-minimal-nilpotent-orbit-closure}
   We have
   \begin{align*}
 X_{V_{k}(\fing)}=\overline{\mathbb{O}_{min}}
\end{align*}
if and only if
   \begin{enumerate}
    \item $\fing$ is of type $A_1$, and $k$ is  a rational admissible
	  number 
	  that is not an integer, or $k=-2$.
    \item $\fing$ is of type $A_2$, $C_l$ $(l\geq 2)$,
	  $F_4$, and $k$ is admissible with denominator $2$.
    \item $\fing$ is of type $G_2$, and $k$ is admissible with denominator $3$, 
    or $k=-1$. 
    \item $\fing$ is of type $D_4$, $E_6$, $E_7$, $E_8$ and
	  $k$ is an integer such that
	  \begin{align*}
	   -\frac{h^{\vee}}{6}-1\leq k\leq -1.
	  \end{align*}
	      \item $\fing$ is of type $D_l$ with $l\geq 5$, and $k=-2,-1$.
   \end{enumerate}
 \end{Conj}
One can easily verify Conjecture
\ref{Conj:when-variety-is-the-minimal-nilpotent-orbit-closure}
for type $A_1$.
Note that
the ``if'' part of Conjecture
\ref{Conj:when-variety-is-the-minimal-nilpotent-orbit-closure}
follows from Theorem \ref{Th1} and \cite[Theorem 5.14]{Ara09b}.

\section{Proof of  Theorem \ref{Th:new-lisse}}
Let 
  $H^{\frac{\infty}{2}+\bullet}_{f_{\theta}}(M)$ denote %the cohomology of 
  the BRST cohomology associated with
the quantized Drinfeld-Sokolov reduction
corresponding  to $f_{\theta}$ (\cite{KacRoaWak03}),
so that
\begin{align*}
\W^k(\fing,f_{\theta})=H^{\frac{\infty}{2}+0}_{f_{\theta}}(V^k(\fing)).
\end{align*}
The correspondence
$M\mapsto  H^{\frac{\infty}{2}+0}_{f_{\theta}}(M)$
gives a functor
$\mathcal{O}_k\ra \W^k(\fing,f_{\theta})\on{-Mod}$,
where $\mathcal{O}_k$ is the category $\mathcal{O}$ of $\affg$ of level
  $k$
  and $\W^k(\fing,f_{\theta})\on{-Mod}$ is the category of $\W^k(\fing,f_{\theta})$-modules.

Recall that $\W_k(\fing,f_{\theta})$ is the unique simple quotient of $\W^k(\fing,f_{\theta})$. 
%Let $L(\lam)\in \mc{O}$ be the irreducible highest weight representation of
%$\affg$ with highest weight $\lam$. 

\begin{Th}\label{Th:previous} 
 \begin{enumerate}
  \item {\rm (\cite[Main Theorem]{Ara05})}
	The functor
	$\mathcal{O}_k\ra \W^k(\fing,f_{\theta})\on{-Mod}$,
	$M\mapsto  H^{\frac{\infty}{2}+0}_{f_{\theta}}(M)$,
	is exact.
  \item {\rm (\cite[Main Theorem]{Ara05})}
We have
$\BRS{L(\lam})
=0$
if $\lam(\alpha_0^{\vee})\in \Z_{\geq 0}$, where $\alpha_0^{\vee}=K-\theta$.
Otherwise 
$\BRS{L(\lam})$  is an irreducible highest weight representation of $\W^k(\fing,f_{\theta})$.
In particular, 
 \begin{align*}
		 H^{\frac{\infty}{2}+0}_{f_{\theta}}(V_k(\fing))
		 \cong
		 \begin{cases}
		   	\W_k(\fing,f_{\theta})&\text{if }
		  k\not\in \Z_{\geq 0}
,\\
		  0&\text{if }k\in \Z_{\geq 0}.
		 \end{cases}
		\end{align*}  
\item {\rm (\cite[Theorem 4.21]{Ara09b})}
For any quotient $V$ of $V^k(\fing)$ 
we have
\begin{align*}
 X_{\BRS{V}}=X_V\cap \mc{S}_{min}.
\end{align*}
Hence
\begin{enumerate}
   \item  {\rm (\cite[Proposition 4.22]{Ara09b})}
	$H^{\frac{\infty}{2}+0}_{f_{\theta}}(V)\ne 0$
	if and only if
	$\overline{\mathbb{O}_{min}}\subset X_{V}$.
  \item  {\rm (\cite[Theorem 4.23]{Ara09b})}
	$H^{\frac{\infty}{2}+0}_{f_{\theta}}(V)$ is a  lisse vertex
	algebra if
	$X_V=\overline{\mathbb{O}_{min}}$.

\end{enumerate} \end{enumerate}
\end{Th}

 \begin{Rem}\label{Rem:when-simples-are-iso}
By \cite[Theorem 6.3]{KacWak04},
the image $\BRS{M(\lam)}$ of the Verma module $M(\lam)$ of $\affg$ with
  highest weight $\lam$ is isomorphic to  a Verma module of
  $\W^k(\fing,f_{\theta})$.
Moreover, all the Verma modules of  $\W^k(\fing,f_{\theta})$ appear in this way.
By  Theorem~\ref{Th:previous}  (1), (2),
$\BRS{L(\lam)}$ is the  unique simple quotient
of $\BRS{M(\lam)}$ 
provided  $\lam(\alpha_0^{\vee})\not \in \Z_{\geq 0}$.
From this, one sees that all the irreducible highest weight
  representations of $\W^k(\fing,f_{\theta})$ appear as 
$\BRS{L(\lam)}$ for some $\lam$, see \cite{Ara05} for the details.

Let $k$ be non-critical, that is,
$k+h^\vee \ne 0$.
By  \cite[Theorem 6.3]{KacWak04},
one finds that 
$\BRS{M(\lam)}\cong \BRS{M(\mu)}$ if and only if 
%\marginpar{\tiny \color{red} Can you fix the typos? I guess 
%we just have to remove "In"?} 
$\mu=s_0\circ \lam$, where $s_0$ is the reflection corresponding to $\alpha_0$.
It follows that
$\BRS{L(\lam)}$ and $\BRS{L(\mu)}$ are nonzero and isomorphic if and
  only if $\lam(\alpha_0^{\vee}),\mu (\alpha_0^{\vee})\not \in \Z_{\geq
  0}$
and $\mu=s_0\circ \lam$.
 \end{Rem}

 \begin{proof}[Proof of Theorem \ref{Th:new-lisse}] 
 Let $k=k_n$ with $n\geq 0$ as in \S \ref{sec:proof-of-th1}.
We have shown that $X_{\tilde{V}_k(\fing)}=\overline{\mathbb{O}_{min}}$
  in
 Proposition \ref{Pro:variety-intermidiate}.
Hence the vertex algebra
 $H^{\frac{\infty}{2}+0}_{f_{\theta}}(\tilde{V}_k(\fing))$ is
 nonzero and lisse by 
  Theorem \ref{Th:previous} (3).
Note that
both
$\W_k(\fing,f_{\theta})$ and
  $  H^{\frac{\infty}{2}+0}_{f_{\theta}}(\tilde{V}_k(\fing))$ are quotients
  of
  $\W^k(\fing,f_{\theta})$.
Indeed,
$  H^{\frac{\infty}{2}+0}_{f_{\theta}}(\tilde{V}_k(\fing))$
is a quotient of 
%\marginpar{\tiny \color{red} Ok?} 
  $\W^k(\fing,f_{\theta})=H^{\frac{\infty}{2}+0}_{f_{\theta}}({V}^k(\fing))$ 
by 
   Theorem \ref{Th:previous} (1)
since
 $\tilde{V}_k(\fing)$ is a quotient of 
  $V^k(\fing)$. 
  Because it is a unique simple quotient of
   $\W^k(\fing,f_{\theta})$,
  $\W_k(\fing,f_{\theta})$ is a quotient of 
  $H^{\frac{\infty}{2}+0}_{f_{\theta}}(\tilde{V}_k(\fing))$, which is lisse
  as we have just proved.
 Therefore $\W_k(\fing,f_{\theta})$ is  lisse as well.
 \end{proof}

 \begin{Conj}Let $\fing$ and $k$ be as in
Theorem \ref{Th:new-lisse}.
Then 
 $\BRS{\tilde{V}_k(\fing)}\cong \W_k(\fing,f_{\theta})$,
where $\tilde{V}_k(\fing)$ is defined above.
 \end{Conj}
 \begin{Rem}
Let $\fing$ and $k$ be as in Theorem \ref{Th:new-lisse}.
Then
$\W_k(\fing,f_{\theta})\not \cong \BRS{L(\lam)}$ for any irreducible admissible
  representation $L(\lam)$ of $\affg$.
Indeed, 
if   $k\leq -1$ (resp.\ if $k\geq -1$),
$
L(k\Lam_0)=V_k(\fing)$ (resp.\ $L(s_0\circ k\Lam_0)$)
is the unique irreducible highest weight representation of $\affg$ such
  that
$\W_k(\fing,f_{\theta})\cong \BRS{L(\lam)}$,
see Remark \ref{Rem:when-simples-are-iso}.
But $k\Lam_0$ (resp.\ $s_0\circ k\Lam_0$) is not an admissible weight
 since it is  not regular dominant.
 \end{Rem}

\section{Classification of lisse minimal $W$-algebras}
 \begin{Th}\label{Th:classfication}
\begin{enumerate}
 \item $\W_k(\mf{sp}_{2l},f_{\theta})$, $l\geq 2$, is lisse if and only if 
$k$ is admissible with denominator $2$, that is,
$k=p/2$ and $p$ is an odd number equal to or greater than $-1$.
 \item $\W_k(\mf{so}_7,f_{\theta})$ is lisse if and only if 
$k$ is admissible with denominator $2$, that is,
$k=p/2$ and $p$ is an odd integer equal to or greater than $-3$.
 \item $\W_k(\mf{so}_{2l+1},f_{\theta})$,
$l\geq 4$, is never lisse. 
 \item $\W_k(\mf{so}_{2l},f_{\theta})$, $l\geq 2$,  is lisse if and only if 
$k$ is an integer
 equal to or greater than $-2$.
 \item $\W_k(F_4,f_{\theta})$ is lisse if and only if 
$k$ is admissible with denominator $2$, that is,
$k=p/2$ and $p$ is an odd number equal to or greater than $-5$.
 \item $\W_k(E_6,f_{\theta})$ is lisse if and only if 
$k$ is an integer
 equal to or greater than $-3$.
 \item $\W_k(E_7,f_{\theta})$ is lisse if and only if 
$k$ is an integer
 equal to or greater than $-4$.
 \item $\W_k(E_8,f_{\theta})$ is lisse if and only if 
$k$ is an integer
 equal to or greater than $-6$.

\end{enumerate}
 \end{Th}

If $\W_k(\fing,f_{\theta})=\C$, then
it is obviously lisse.
Hence it is natural to ask when $\W_k(\fing,f_{\theta})=\C$.
It turns out not every $W$-algebra admits one-dimensional
representations.

\begin{Th}\label{TH:trivial-rep}
Suppose $\fing$ is not of type $A_1$.
The following are equivalent:
\begin{enumerate}
 \item $\W^k(\fing,f_{\theta})$ admits a (non-twisted or Ramond-twisted)
       one-dimensional representation,
\item  $\W_k(\fing,f_{\theta})=\C$,  
\item
\begin{enumerate}
 \item $\fing$ belongs to the Deligne exceptional series 
and
$k=-h^{\vee}/6-1$, or 
\item $\fing=\mf{sp}_{2l}$, $l\geq 2$, and $k=-1/2$.
%\marginpar{\tiny Is
%      part (2) correct?
%{\color{red}Good question. I think it is correct, see Remark \ref{Rem:typeC}} 
%I don't really understand why this remark proves the theorem for 
%$\fing=\mf{sp}_{2l}$, no it doesn't.}
\end{enumerate}
\end{enumerate}
\end{Th}

\begin{Rem} 
If $\fing=\mf{sl}_2$, then $f_\theta=f_{\rm reg}$ is regular, $\W_k(\fing,f_\theta)=\W_k(\mf{sl}_2,f_{\rm reg})$ 
is the simple Virasoro vertex algebra provided that $k\not=-2$, and the results are well-known\footnote{Note 
that $\W_{r-2}(\mf{sl}_2,f_{\rm reg}) \cong \W_{1/r-2}(\mf{sl}_2,f_{\rm reg})$ for any $r\in\C^*$.}. Namely,  
\begin{itemize}
\item[-] $\W_k(\mf{sl}_2,f_{\rm reg})$ is lisse if and only if either $k+2=p/q$, with $p,q\in \Z_{\geq 0}$, $(p,q)=1$ 
and $p,q\geq 2$, or $k+2=0$ (cf.~\cite{Ara12}), 
\item[-] $\W_k(\mf{sl}_2,f_{\rm reg})=\C$ if and only if either $k+2=2/3$, or $k+2=3/2$, or $k+2=0$.
\end{itemize}
\end{Rem}

The rest of this section is devoted to the proof of Theorem
\ref{Th:classfication}
and
Theorem \ref{TH:trivial-rep}.

Let $\fing_0$ be the center of the reductive Lie algebra
$\fing^{\natural}$,
so that
\begin{align*}
\fing^{\natural}=\bigoplus_{i\geq 0} \fing_i.
\end{align*}

Define an invariant bilinear form
on $\fing_i$, $i\geq 0$, by
\begin{align*}
 (x|y)^{\natural}_i:=(k+\frac{h^{\vee}}{2})(x|y)-\frac{1}{4}\left(\on{tr}_{\fing(0)}(\ad x \ad
 y)
\right),
\end{align*}
where $(~|~)$ is the normalized inner product of $\fing$ as before.
Then there exists a polynomial 
$k^{\natural}_i$ of $k$  of degree $1$
such that
\begin{align*}
 (~|~)^{\natural}_i=k^{\natural}_i(~|~)_i,
\end{align*}
where  $(~|~)_i$ is the normalized inner product of $\fing_i$,
that is,
$(\theta_i|\theta_i)=2$.

By \cite[Theorem 5.1]{KacWak04},
we have an embedding
\begin{align*}
 \bigotimes_{i\geq 0} V^{k^{\natural}_i}(\fing_i)
 \hookrightarrow \W^k(\fing,f_{\theta})
\end{align*}
of vertex algebras.

 \begin{Lem}\label{Lem:necessary-condition-for-lisse}
\begin{enumerate}
\item Suppose that $\W_k(\fing,f_{\theta})$ is lisse.
Then the value of $k_i^{\natural}$ for all $i\geq 1$ 
must be a nonnegative integer.
\item Suppose that $\W^k(\fing,f_{\theta})$ admits a (non-twisted or Ramond-twisted)
one-dimensional representation.
Then the value of $k_i^{\natural}$ for all $i\geq 0$ 
must be zero.
\end{enumerate}
 \end{Lem}
 
 \begin{proof}
 (1) By \cite{DonMas06}, if a lisse vertex algebra $V$ contains a quotient
of an affine vertex algebra as a vertex subalgebra, this quotient must be integrable. 
With $V=\W_k(\fing,f_{\theta})$, we deduce that  the simple quotient 
$V_{k^{\natural}_i}(\fing_i)$ 
must be integrable for any $i \geq 1$, that is, 
$k_i^{\natural}$ is a nonnegative integer for any $i\geq 1$. 
%\marginpar{\tiny I think also the second claim "Hence,..." 
%follows from \cite{DonMas06}, all right? yes.}
%Hence all $k_i^{\natural}$, $i\geq 1$, must be a nonnegative integer
%if $\W_k(\fing,f_{\theta})$ is lisse.

(2) If  $\W^k(\fing,f_{\theta})$ admits a (non-twisted or Ramond-twisted)
  one-dimensional representation,
by restriction we obtain that $V^{k_i^\natural}(\fing_i)$, for $i\geq 0$, 
admits a one-dimensional representation. Hence $k_i^{\natural}=0$ 
for all $i\geq 0$.
 \end{proof}

\begin{Lem}\label{Lem:shifted-levels} 
The reductive Lie algebras $\fing^\natural=\bigoplus_{i\geq 0} \fing_i$ 
and the polynomials $k_i^\natural$ are described in the below Tables 
\ref{tab:k} and \ref{tab2:k}. 
%\marginpar{\tiny{thanks for making it into a table!}}
{\begin{table}[h] \tiny 
\begin{center}
\begin{tabular}{|l|c|c|c|c|c|c|}
\hline &&&&&& \\[-0.5em]
& $\mf{sl}_3$ & $\mf{sl}_{l+1}$, $l\geq 3$ & $\mf{sp}_{2l}$, $l\geq 2$ 
& $\mf{so}_{7}$ & $\mf{so}_8$ & $\mf{so}_{n}$, $n\geq 9$ \\[0.25em]
\hline &&&&&&\\[-0.5em]
$\fing^{\natural}$& $\fing_0$, & $\fing_0\+ \fing_1$ & 
$\fing_1$, & $\fing_1\+\fing_2$,
& $\bigoplus_{i=1}^3\fing_i$,
&$\fing_1\+\fing_2$,\\[0.25em]
& $\fing_0\cong\C$ & $\fing_0\cong\C$, $\fing_{1}\cong\mf{sl}_{l-1}$  & 
$\fing_1\cong\mf{sp}_{2l-2}$ & 
$\fing_1\cong \fing_2\cong \mf{sl}_2$ & 
$\fing_i\cong \mf{sl}_2$,
&
$\fing_1\cong \mf{sl}_2$,
$\fing_2\cong \mf{so}_{n-4}$ \\[0.25em]
\hline &&&&&&\\[-0.5em]
$k_i^{\natural}$ & $ k_0^{\natural}=k+\frac{3}{2}$ & $k_0^{\natural}=k+\frac{l+1}{2},$ 
 & $k_1^{\natural}=
k+\frac{1}{2}$ & $ k_1^{\natural}=k+\frac{3}{2}$,  & $ k_i^{\natural}=k+2$, 
& $k_1^{\natural}
 =k+\frac{n}{2}-2,$ 
\\[0.25em]
&  & 
$k_1^{\natural}=k+1$ &  & 
$k_2^{\natural}=2k+4$ &  $i\in\{1,2,3\}$ 
& $k_2^{\natural}
 =k+2$ \\[0.25em]
 \hline
\end{tabular}
\vspace{.25cm}
\caption{$\fing^{\natural}=\bigoplus_{i\geq 0} \fing_i$ 
and $k_i^\natural$ for the classical types} \label{tab:k}
\end{center}
\end{table}}
{\begin{table}[h] \tiny 
\begin{center}
\begin{tabular}{|l|c|c|c|c|c|}
\hline &&&&& \\[-0.5em]
& $G_2$ & $F_4$ & $E_6$ 
& $E_7$ & $E_8$ \\[0.25em]
\hline &&&&& \\[-0.5em]
$\fing^{\natural}$
& $\mf{sl}_2$ 
& $\mf{sp}_6$ 
& $\mf{sl}_6$ 
& $\mf{so}_{12}$ 
& $E_7$ 
\\[0.5em]
$k_1^{\natural}$ 
& $3k+5$ 
& $k+\frac{5}{2}$ 
& $k+3$
& $k+4$ 
&  $k+6$ 
\\[0.25em]
\hline 
\end{tabular}
\vspace{.25cm}
\caption{$\fing^{\natural}=\bigoplus_{i\geq 0} \fing_i$ 
and $k_i^\natural$ for the exceptional types} \label{tab2:k}
\end{center}
\end{table}}

 \end{Lem}
 
 \begin{proof} 
 The verifications are easy and left to the reader.
 \end{proof}

 \begin{proof}[Proof of Theorem \ref{Th:classfication}]
The ``if'' part of Theorem \ref{Th:classfication}
has been already proven in
Theorem \ref{Th:new-lisse}
and \cite[Theorem 5.18]{Ara09b}, 
and the ``only if'' part follows from Lemmas
\ref{Lem:necessary-condition-for-lisse} and \ref{Lem:shifted-levels}.
 \end{proof}

 \begin{Rem}\label{Rem:typeC}
For $\fing=\mf{sp}_{2l}$ it is possible to show the following.
%\color{red}
\begin{align*}
 A(V_{-1/2}(\fing))\cong U(\fing)/\mc{J}_{W_1}\cong
 \C\times (L_{\fing}(\varpi_1)^*\otimes_{\C} L_{\fing}(\varpi_1))
 \times U(\fing)/\mc{J}_0,
\end{align*}
%\color{black}
where $\mc{J}_{W_1}$ is the ideal generated by
  $W_1:=L_{\fing}(\theta+\theta_1)\subset W$.
%(At moment I can prove this only by using affine KM algebras!)
This implies that $\mc{J}_0$ is generated by $\mc{J}_{W_1}$ and 
$\Omega-c_0$.
% but I am not sure whether this is useful. 
%\marginpar{\tiny What shall we do with this remark??}
 \end{Rem}
 \begin{Conj}\label{Conj:remaining-classification}
\begin{enumerate}
\item
$\W_k(\mf{sl}_3,f_{\theta})$ is lisse if and only if 
$k$ is admissible with denominator $2$,
that is,
$k=p/2$ and $p$ is an odd integer equal or greater than $-3$.

	      \item $\W_k(\mf{sl}_n,f_{\theta})$,
$n\geq 4$,
is never lisse.
\item $\W_k(G_2,f_{\theta})$ is lisse if and only if $k$ is admissible with
      denominator $3$, 
      or an  integer equal to or greater than $-1$. 
	     \end{enumerate}
 \end{Conj}
The ``if'' part of Conjecture \ref{Conj:remaining-classification}
follows from \cite[Theorem 5.18]{Ara09b}.

 \begin{proof}[Proof of Theoren \ref{TH:trivial-rep}]
Clearly (2) implies (1).
The direction (1) $\Rightarrow$ (3)
follows from Lemmas 
\ref{Lem:necessary-condition-for-lisse} and \ref{Lem:shifted-levels}.

Let us show (3) implies (2).

The $A_2$ case follows from \cite{Ara10BP}.

Assume that $\fing$ is of type $D_l$, $E_6$,
$E_7$, or $E_8$.
Note that
 $k=k_0$ in \eqref{eq:value-of-k}.
Let $N$ be the submodule of $V^k(\fing)$ generated by
$v_i=\sigma(w_i)$, for all $i$, 
and set $\tilde{V}_{k}(\fing)=V^k(\fing)/N$
 as in Section \ref{sec:proofs}.
By Theorem \ref{Th:previous} (1) 
we have an exact sequence
\begin{align*}
 0\ra \BRS{N}\ra \BRS{V^k(\fing)}
\ra \BRS{\tilde{V}_k(\fing)}\ra 0
\end{align*}
of 
$\W^k(\fing,f_{\theta})$-modules.
The image $\bar{v}_i$ of $v_i\in N$  in 
$ \BRS{V^k(\fing)}=\W^k(\fing,f_{\theta})$
is nonzero,
since  its image in $R_{\W_k(\fing,f_{\theta})}=\C[\mc{S}_{min}]$ is nonzero
and coincides with 
 $e_{\theta_i}$ under the identification
$\C[\mc{S}_{min}]=S(\fing^{f_{\theta}})$,
where
$e_{\theta_i}$ is the highest root vector of $\fing_i$.
By weight consideration one finds that
$\bar{v}_i$  
%has conformal weight $1$
%kand
coincides with
$e_{\theta_i}(-1)\in V^{k_i^{\natural}}(\fing_i)\subset \W^k(\fing,f_{\theta})$
up to non-zero constant multiplication.

Since
$\W^k(\fing,f_{\theta})_1=\fing^{\natural}=\bigoplus_{i\geq 1}\fing_i$,
the whole weight one space $\W^k(\fing,f_{\theta})_1$ is included in the
image of $\BRS{N}$.
Then
 from the commutation relations of $\W^k(\fing,f_{\theta})$
  described in 
\cite[Theorem 5.1]{KacWak04}
 it follows that
all the  generators $G^{v}$, 
$v\in \fing_{1/2}$,
defined in \cite{KacWak04},
and
 the conformal vector
are also
in the image of $\BRS{N}$.
Therefore
$\BRS{\tilde{V}_k(\fing)}$ must be trivial, and hence, so is
its simple quotient
  $\W_k(\fing,f_{\theta})$.

Assume that
$\fing$ is of type  $C_l$, $G_2$ or $F_4$, so that
$\fing^{\natural}$ is simple
and
$k$ is admissible,
and hence the maximal submodule
$N_k$
of 
$V^k(\fing)$ is generated by a singular vector $v$.
By Theorem \ref{Th:previous} (1), (2) we have the exact sequence
\begin{align*}
0\ra \BRS{N}
\ra \W^k(\fing,f_{\theta})\ra \W_k(\fing,f_{\theta})\ra 0.
\end{align*}
Also, by \cite[Theorem 6.3.1]{KacWak04}
$\BRS{N}$ is 
generated by the image $\bar v$ of $v$.
Since the image of $\BRS{N}$ in $\W^k(\fing,f_{\theta})$ is nonzero
as $\W_k(\fing,f_{\theta})$ is lisse \cite{Ara09b},
the image of $\bar v$ in $\W^k(\fing,f_{\theta})$
is nonzero.
Hence, as in the same manner as above,
by weight consideration
it follows that
$\W^1(\fing,f_{\theta})_1$ is included in the image of $\BRS{N}$,
which gives that
$\W_k(\fing,f_{\theta})=\C$ 
as required.
 \end{proof}
\newcommand{\etalchar}[1]{$^{#1}$}

%\bibliographystyle{alpha}

%\bibliography{/Users/tomoyuki/Documents/Dropbox/bib/math}

\end{document}